\documentclass[a4paper]{article}
\usepackage{graphicx} 
\usepackage{authblk}
\usepackage[T1]{fontenc}
\usepackage{hidde-standard-preamble} 
\usepackage{thm-restate}

\newcommand{\atw}{\alpha\text{-}\mathsf{tw}}
\newcommand{\tw}{\mathsf{tw}}

\DeclareEmphSequence{\itshape\color{blue}}

\title{Tree-independence number of $K_{1,d}$-free graph classes}
\author[1,2]{Kenny Be\v{s}ter \v{S}torgel\thanks{Supported in part by the Slovenian Research and Innovation Agency, research program P1–0383 and research project N1-0370}}
\author[3,4]{Mujin Choi\thanks{Supported by the Institute for Basic Science (IBS-R029-C1)}}
\author[5]{Hidde Koerts}
\author[6]{Đorđe Vasić}

\affil[1]{Faculty of Information Studies in Novo mesto, Slovenia}
\affil[2]{FAMNIT, University of Primorska, Slovenia}
\affil[3]{Department of Mathematical Sciences, KAIST, South Korea}
\affil[4]{Discrete Mathematics Group, Institute for Basic Science, KAIST, South Korea}
\affil[5]{Department of Combinatorics and Optimization, University of Waterloo, Canada}
\affil[6]{Faculty of Mathematics and Physics, University of Ljubljana, Slovenia}

\begin{document}

\maketitle
\renewcommand{\thefootnote}{}
\footnotetext{Emails: kenny.storgel@fis.unm.si, mujinchoi@kaist.ac.kr, hkoerts@uwaterloo.ca, xv84964@student.uni-lj.si}

\begin{abstract}
    In this paper, we investigate the tree-independence number of graph classes that do not contain $K_{1,d}$ as an induced subgraph.
    Dallard et al. conjectured that for any positive integer $d$ and any planar graph $H$, the class of all $K_{1,d}$-free graphs without $H$ as an induced minor has bounded tree-independence number.
    Our main contribution towards this conjecture is showing that the conjecture holds for outerstring graphs.
    Additionally we give linear and quadratic bounds for the tree-independence number of various $K_{1,d}$-free graph classes, sharpening previous bounds. 
    Finally, we bound the tree-independence number of $K_{2,d}$-free graphs additionally forbidding holes of length at least $5$.

\end{abstract}

\section{Introduction}\label{sec:introduction}

In this paper we study a relatively recent graph width parameter called the \emph{tree-independence number}, denoted by $\alpha\text{-}\mathrm{tw}$ in certain graph classes excluding a star (i.e., a graph $K_{1,d}$ for some integer $d$) as an induced subgraph. Tree-independence number was introduced independently by Yolov~\cite{MR3775804} and Dallard, Milani\v{c}, and \v{S}torgel~\cite{DBLP:journals/jctb/DallardMS24}.
Roughly speaking, it measures how similar a graph is to being chordal.
It is defined similarly as treewidth, except that instead of measuring the cardinality of each bag of a tree decomposition (minus $1$), we measure the independence number of the graph induced by the set of vertices contained in each bag.
It is known that boundedness of tree-independence number gives a sufficient condition for polynomial-time solvability of many algorithmic problems related to independent sets, including the \textsc{Maximum Weight Independent Packing} problem and, as a consequence, the \textsc{Maximum Weight Independent Set} and the \textsc{Maximum Weight Induced Matching} problems (see \cite{DELIMA2026,DBLP:journals/jctb/DallardMS24,MR3775804}).

Given two graphs $G$ and $H$, we say that $G$ is \emph{$H$-free} if it does not contain an induced subgraph isomorphic to~$H$. Moreover, for a class $\mathcal{F}$ of graphs, we say that $G$ is \emph{$\mathcal{F}$-free} if for every graph $F\in \mathcal{F}$, $G$ is $F$-free. Similarly, we say that a graph class $\mathcal{G}$ is $H$-free if every graph $G\in \mathcal{G}$ is $H$-free, and $\mathcal{G}$ is $\mathcal{F}$-free if every graph $G\in \mathcal{G}$ is $\mathcal{F}$-free.

Recently, a result giving a complete characterisation of graph classes with bounded tree-independence number, defined by forbidding a finite set of graphs as an induced subgraph, was shown by Hajebi and Spirkl~\cite{hajebi2026treealphaexcludingfinitelygraphs}.
Let $\mathcal{S}$ be the family of all graphs whose every connected component is a tree with at most three leaves.

\begin{theorem}[Hajebi and Spirkl~\cite{hajebi2026treealphaexcludingfinitelygraphs}]
    \label{thm:hajebi-spirkl}
    Let $\mathcal{F}$ be a finite set of graphs. Then the class of $\mathcal{F}$-free graphs has bounded tree-independence number if and only if $\mathcal{F}$ contains a complete bipartite graph, a graph from $\mathcal{S}$ and a line graph of a graph from $\mathcal{S}$.
\end{theorem}

In addition, \cref{thm:hajebi-spirkl} confirms the validity of several conjectures about tree-independence number stated by Dallard et al. in~\cite{dallard2024treewidthversuscliquenumber4}, including a conjecture that the class of $\{P_k, K_{d,d}\}$-free graphs has bounded tree-independence number.

On the other hand, by a result of Chudnovsky and Trotignon~\cite{ChudnovskyTrotignon2025}, \cref{thm:hajebi-spirkl} does not extend to the setting of hereditary classes defined by infinite sets of forbidden induced subgraphs. 

Note that the above results do not give much insight in terms of the upper bound on the tree-independence number for particular graph classes neither in terms of finite nor infinite sets of forbidden induced subgraphs.

The aim of this paper is to focus on specific upper bounds for various graph classes. Specifically, we will consider graph classes forbidding $K_{1,d}$ as an induced subgraph.
The importance of $K_{1,d}$-free graph classes can be seen as follows. Korhonen~\cite{KorhonenGrid} showed that an induced variant of the celebrated Grid-Minor Theorem by Robertson and Seymour~\cite{GraphMinorV} holds for graphs with bounded maximum degree. 

\begin{theorem}[\cite{KorhonenGrid}]
    \label{thm:korhonen}
    There exists a function $f: \mathbb{N}^2\rightarrow \mathbb{N}$ such that for every positive integer $k$ and every graph $G$, if $\tw(G) > f(\Delta(G), k)$, then $G$ contains the $(k\times k)$-grid as an induced minor.
\end{theorem}

In their paper, Dallard et al.~\cite{dallard2024treewidthversuscliquenumber4} posed a conjecture that, if true, would generalise \cref{thm:korhonen}.

\begin{conjecture}[\cite{dallard2024treewidthversuscliquenumber4}]\label{conj:inducedgridequiv}
    Let $d$ be a positive integer and let $\mathcal{G}$ be a hereditary $K_{1,d}$-free graph class.
    Then $\mathcal{G}$ has bounded tree-independence number if and only if there exists a positive integer $k$ such that $\mathcal{G}$ excludes the $(k\times k)$-grid as an induced minor.
\end{conjecture}

Chudnovsky et al.~\cite{arXiv:2512.23887} showed that for any $d$ and every planar graph $H$, the class $\mathcal{G}$ of $H$-induced-minor-free $K_{1,d}$-free graphs has polylogarithmic tree-independence number.
Additionally, it is known that the broader class of $H$-induced-minor-free graphs has a bounded tree-independence number for the following cases.

\begin{theorem}[\cite{dallard2024treewidthversuscliquenumber3}]
    Let $H$ be a graph and let $\mathcal{G}$ be the class of $H$-induced-minor-free graphs. Then for every graph $G\in \mathcal{G}$ the following holds:
    \begin{itemize}
    \item $\atw(G)\le 1$ if $H = C_4$ (i.e., $G$ is chordal),
    \item $\atw(G)\le 4$ if $H\in \{W_4, K_5^{-}\}$, and
    \item $\atw(G)\le 2q-2$ if $H = K_{2,q}$.
\end{itemize}
\end{theorem}

Moreover, in the case when $\mathcal{G}$ is the class of $K_{1,d}$-free graphs, then it is known that $\mathcal{G}$ has bounded tree-independence number for the following more restricted cases.

\begin{theorem}
    Let $H$ be a graph and let $\mathcal{G}$ be the class of $K_{1,d}$-free $H$-induced-minor-free graphs. Then,
    \begin{itemize}
    \item if $H$ is the graph $kC_3$, then $\alpha\text{-}\mathrm{tw}(G)= \mathcal{O}(dk\log k)$ ~\cite{ahn2024coarseerdhosposatheorem,arXiv:2505.12866}; 
    \item if $H$ is a wheel $W_{\ell}$, then $\alpha\text{-}\mathrm{tw}(G) \le 8\ell(d-1) + 8d - 15$~\cite{choi2025excludinginducedwheelminor,choi2025excludingladderinducedminor};
    \item if $H$ is a $k$-ladder, then $\atw(G)\leq 64(d-1)\cdot (2d)^{2\cdot(2d)^{(2d)^{2dk^2+1}}}+8d-14$~\cite{choi2025excludingladderinducedminor}.
\end{itemize}
\end{theorem}


The paper is organised as follows.
In \cref{sec:prelim} we present some important definitions and notation. Then, in \cref{sec:without_long_cycle} we investigate several $K_{1,d}$-free graph classes, namely the classes of tolerance, co-comparability, and AT-free graphs, as well as a particular case of bisplit graphs, and show that they have bounded tree-independence number with a linear upper bound in terms of $d$. In \cref{sec:outerstring} we then prove that classes of $K_{1,d}$-free outerstring and, more generally, $k$-outerstring graphs have bounded tree-independence number by showing some explicit upper bounds. We continue in \cref{sec:unbounded} by showing that classes of $K_{1,d}$-free circle containment graphs and rook graphs have unbounded tree-independence number. Several results from Sections~\ref{sec:without_long_cycle}, \ref{sec:outerstring}, and~\ref{sec:unbounded} are captured in \cref{fig:placeholder}.
In \cref{sec:K2d} we then turn our attention to $K_{2,d}$-free graphs and generalise some results from \cref{sec:without_long_cycle} by showing that when excluding all cycles of length at least $5$ such graph classes have bounded tree-independence number, again with a linear upper bound in terms of $d$. Finally, in \cref{sec:conclusion} we present a couple of open problems.



\begin{figure}
    \centering
    \includegraphics[width=0.667\linewidth]{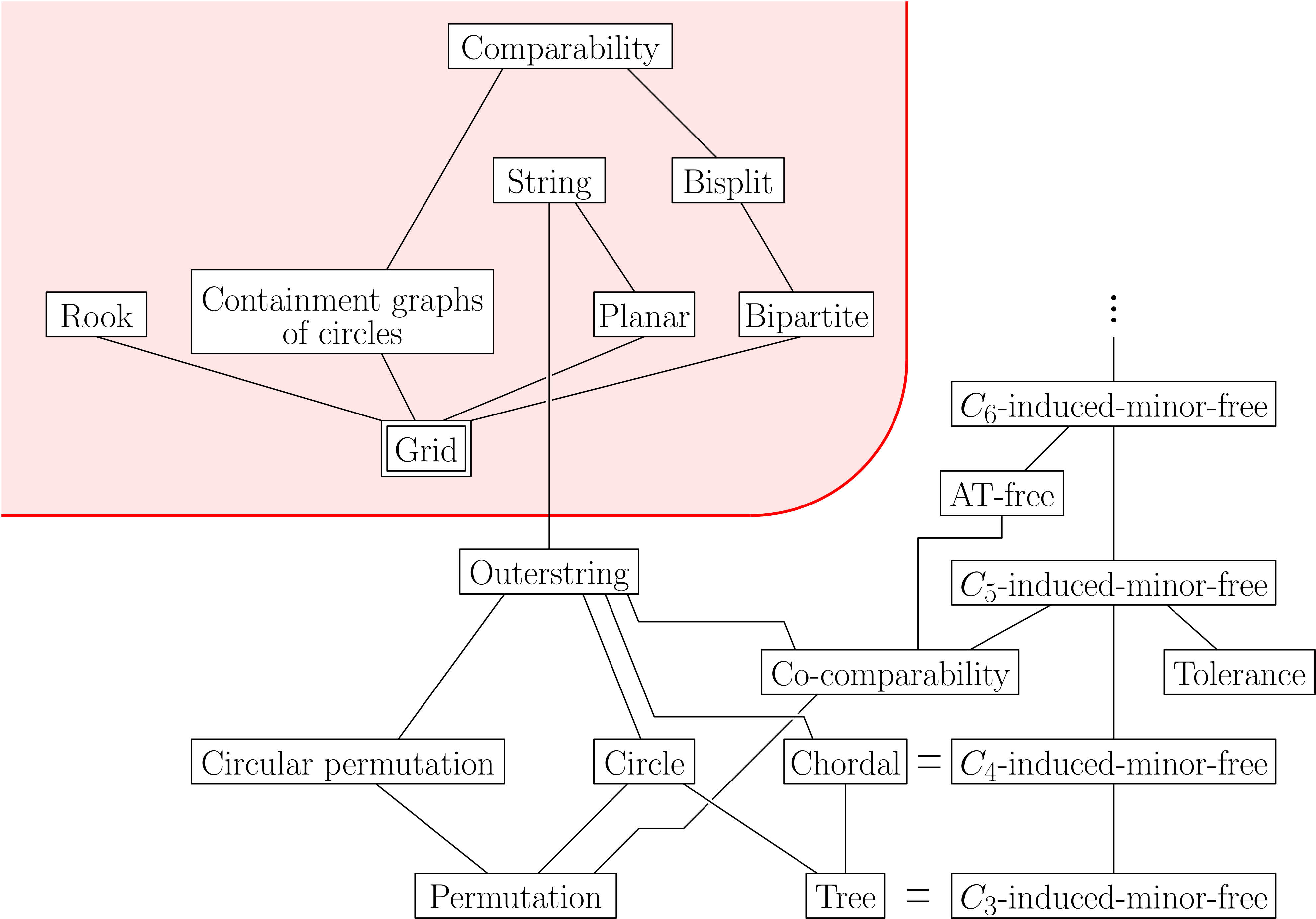}
    \caption{Containment relations of graph classes in our paper. Graph classes contained in the red part have unbounded tree-independence number even when $K_{1,d}$ is forbidden as an induced subgraph while the others have bounded tree-independence number.
    }
    \label{fig:placeholder}
\end{figure}

\section{Preliminaries}\label{sec:prelim}

All graphs considered in this paper are finite, simple (i.e., without loops or parallel edges), and undirected. We write $V(G)$ and $E(G)$ for the vertex set and edge set of a graph $G$, respectively. 


For $v\in V(G)$, the set $N_G(v) = \{ u\in V(G) : vu\in E(G)\}$ is the \emph{neighbourhood} of $v$,  and $N_G[v] = N(v) \cup \{v\}$ is the \emph{closed neighbourhood} of $v$.
Similarly, for $S\subseteq V(G)$, we write $N_G[S]=\bigcup_{v\in S}N_G[v]$ and $N_G(S) = N_G[S]\setminus S$.
The \emph{degree} of a vertex $v\in V(G)$, denoted by $d_G(v)$, is the number of edges incident with $v$, or equivalently, $d_G(v)=|N(v)|$.
If the graph is clear from the context, we may just write $N(v)$, $N[v]$, $N(S)$, $N[S]$ and $d(v)$ instead of $N_G(v)$, $N_G[v]$, $N_G(S)$, $N_G[S]$ and $d_G(v)$, respectively.

An \emph{independent set} in a graph $G$ is a set of pairwise nonadjacent vertices.
The \emph{independence number} of a graph $G$, denoted by $\alpha(G)$, is the cardinality of a largest independent set in $G$.
A \emph{clique} in a graph $G$ is a set of vertices such that every two distinct vertices are adjacent.
The \emph{clique number} of a graph $G$, denoted by $\omega(G)$, is the size of a largest clique in $G$.

A graph $H$ is an \emph{induced subgraph} of a graph $G$ if a graph isomorphic to $H$ can be obtained from $G$ by a sequence of vertex deletions.
Given a set $S\subseteq V(G)$, we denote by $G[S]$ the subgraph of $G$ \emph{induced} by $S$.
For every set $S\subseteq V(G)$, we denote by $G-S$ the graph $G[V(G)\setminus S]$, and similarly, for every $u\in V(G)$, $G-u$ corresponds to the graph $G[V(G)\setminus \{u\}]$.
Given a graph $G$ and an edge $uv$ in $G$, \emph{contracting} the edge $uv$ means replacing $u$ and $v$ with a single new vertex $w$ so that the vertices adjacent to $w$ are exactly the vertices of $G$ that are adjacent to $u$ or $v$.
A graph $H$ is an \emph{induced minor} of a graph $G$ if a graph isomorphic to $H$ can be obtained from $G$ by a sequence of vertex deletions and edge contractions.
If a graph $G$ does not contain a graph $H$ as an induced minor, then we say that $G$ is \emph{$H$-induced-minor-free}.

A \emph{complete graph} of order $n$ is denoted by $K_n$.
Given two integers $m,n\ge 0$, the \emph{complete bipartite graph} $K_{m,n}$ is a bipartite graph with parts of size $m$ and $n$, such that two vertices are adjacent if and only if they belong to different parts.
The path graph and the cycle graph with $n$ vertices are denoted by $P_n$ and $C_n$, respectively.
Given a graph $G$ and an integer $k\ge 0$, a \emph{path of length $k$ in} $G$ is a sequence $P = v_1\ldots v_{k+1}$ of pairwise distinct vertices of $G$ such that $v_iv_{i+1}\in E(G)$ for all $i = 1,\ldots, k$.
A path $v_1\ldots v_k$ with $k\ge 3$ and $v_kv_1\in E(G)$ is a \emph{cycle of length $k$ in} $G$.
An edge $v_iv_j\in E(G)$ with $|i-j|>1$ is a \emph{chord} of the cycle; the cycle is said to be \emph{chordless} if it has no chords.
Given two disjoint subsets of vertices $A$ and $B$, a set $S$ is a \emph{separator} between $A$ and $B$ in a graph $G$ if there is no path between a vertex from $A$ and a vertex from $B$ in the graph $G - S$.
We remark that possibly $S\cap A\neq \emptyset$ or $S\cap B\neq \emptyset$.

A \emph{tree decomposition} of a graph $G$ is a pair $\mathcal{T}=(T, \{ X_t \}_{t\in V(T)})$ where $T$ is a tree and every node $t$ of $T$ is assigned a vertex subset $X_t\subseteq V(G)$ called a \emph{bag} such that the following conditions are satisfied: every vertex~$v$ is in at least one bag, for every edge $uv\in E(G)$ there exists a node $t\in V(T)$ such that $X_t$ contains both $u$ and $v$, and for every vertex $u\in V(G)$ the subgraph $T_u$ of $T$ is induced by the set $\{ t\in V(T) : u\in X_t \}$ is connected, (that is, a tree).
The \emph{width} of $\mathcal{T}$, denoted by $width(\mathcal{T})$, is the maximum value of $|X_t|-1$ over all $t\in V(T)$.
The \emph{treewidth} of a graph $G$, denoted by $\tw(G)$, is defined as the minimum width of a tree decomposition of $G$.



The \emph{independence number} of $\mathcal{T}$ is defined as
\[\alpha(\mathcal{T}) = \max_{t\in V(T)} \alpha(G[X_t])\,.\]
The \emph{tree-independence number} of $G$, denoted by $\atw (G)$, is the minimum independence number among all possible tree decompositions of $G$.

\section{Graph classes without long induced cycles}\label{sec:without_long_cycle}

In this section we consider the question of establishing an upper bound on the tree-independence number for specific $K_{1,d}$-free graph classes excluding long induced cycles. One particular result in this direction is the following result by Dallard et al.~\cite{dallard2024treewidthversuscliquenumber4}.

\begin{theorem}[\cite{dallard2024treewidthversuscliquenumber4}]
    \label{thm:PkK1d}
    Let $k\ge 3$ and $d\ge 2$ be integers and let $G$ be a $\{P_k, K_{1,d}\}$-free graph. Then $\alpha\text{-}\mathrm{tw}(G)\le (d-1)(k-2)$.
\end{theorem}

Moreover, one can easily obtain the following slight generalisation of \cref{thm:PkK1d} (see also \cite{dallard2024treewidthversuscliquenumber4,choi2025excludinginducedwheelminor}).

\begin{theorem}
    \label{thm:atw_K1d-free_longest_cycle}
    Let $\ell\ge 3$ and $d\ge 2$ be integers and let $G$ be a $K_{1,d}$-free graph with no hole of length more than $\ell$. Then $\alpha\text{-}\mathrm{tw}(G)\le (d-1)(\ell - 2)$.
\end{theorem}

Note that in the above theorem, if $\ell = 3$, then $G$ is a chordal graph and the tree-independence number of chordal graphs is equal to $1$. 

On another note, we already mentioned that \cref{thm:hajebi-spirkl} yields boundedness of the class of $\{P_k, K_{d,d}\}$-free graphs. Although not much is known in terms of the upper in the general case, a recent result by Blažej et al.\cite{blazej2026treeindependencenumberp5freegraphs} shows that $\{P_5, K_{d,d}\}$-free graphs have tree-independence number at most $4d$.

We now explore the tree-independence number of certain graph classes without long induced cycles. 
First, we present several examples of $K_{1,d}$-free graph classes for which the tree-independence number is bounded as an immediate consequence of Theorem~\ref{thm:atw_K1d-free_longest_cycle}.

\subsection{Tolerance, co-comparability, and AT-free graphs}

An \emph{interval graph} is a graph where each vertex $v$ can be assigned a closed interval $I_v$ on the real line, such that $x$ and $y$ are adjacent if and only if $|I_x \cap I_y | \neq \emptyset$.
Moreover, if every vertex $v$ of an interval graph can be assigned a positive real value $t_v$ (called a \emph{tolerance}) such that $x$ and $y$ are adjacent if and only if $|I_x \cap I_y | \ge \min \{t_x , t_y \}$, then the graph is said to be a \emph{tolerance graph}.

\begin{corollary}\label{no_long_hole_tolerance}
     Let $G$ be a $K_{1,d}$-free tolerance graph with $d\geq 2$ an integer. Then $\alpha\text{-}\mathsf{tw}(G)\leq 2d-2$.
\end{corollary}

\begin{proof}
     As shown by Golumbic, Monma, and Trotter~\cite[Theorem 2]{GOLUMBIC1984157}, tolerance graphs do not contain holes of length at least $5$. Thus, by \cref{thm:atw_K1d-free_longest_cycle}, $\alpha\text{-}\mathsf{tw}(G)\leq 2d-2$.
\end{proof}

A graph is \emph{comparability} if its edges can be oriented such that if $(a,b)$ and $(b,c)$ are directed edges, then there is a directed edge $(a,c)$.
A graph $G$ is \emph{co-comparability} if $\overline{G}$ is comparability.

\begin{corollary}\label{no_long_hole_co-comparability}
    Let $G$ be a $K_{1,d}$-free co-comparability graph with $d\geq 2$ an integer. Then $\alpha\text{-}\mathsf{tw}(G)\leq 2d-2$.
\end{corollary}

\begin{proof}
    As shown by Gallai~\cite{GallaiTransitive}, comparability graphs contain no complements of holes of length at least $5$.
    Hence co-comparability graphs do not contain holes of length at least $5$. Thus, by \cref{thm:atw_K1d-free_longest_cycle}, $\alpha\text{-}\mathsf{tw}(G)\leq 2d-2$.
\end{proof}

A triple $x,y,z$ of vertices in a graph $G$ is called an \emph{asteroidal triple} if between each pair of them there exists a path in $G$ that avoids the closed neighbourhood of the third one.
A graph is \emph{AT-free} if it does not contain an asteroidal triple.

\begin{corollary}\label{no_long_hole_AT-free}
    Let $G$ be a $K_{1,d}$-free AT-free graph with $d\geq 2$ an integer. Then $\alpha\text{-}\mathsf{tw}(G)\leq 3d-3$.
\end{corollary}

\begin{proof}
    AT-free graphs do not contain holes of length at least $6$, since three mutually non-adjacent vertices in holes of length at least $6$ form an asteroidal triple.
    Thus, by \cref{thm:atw_K1d-free_longest_cycle}, $\alpha\text{-}\mathsf{tw}(G)\leq 3d-3$.
\end{proof}

\subsection{Bisplit graphs}

A graph is \emph{bisplit} if there exists a partition $(X,Y,Z)$ of its vertices such that $X$, $Y$, and $Z$ are independent sets and $Y\cup Z$ induces a complete bipartite graph. We call such a partition a \emph{bisplit partition}. Since walls are bipartite and the class of bisplit graphs contains all bipartite graphs, it follows that the class of bisplit graphs contains $K_{1,4}$-free graphs of arbitrarily large tree-independence number. However, adding one extra requirement that both sides of the biclique must be non-empty dramatically decreases the tree-independence number of the class.

First, we prove that adding the above mentioned requirement results in a bisplit graph whose every induced cycle has bounded length.

\begin{lemma}
    \label{lem:bisplit-cycle-free}
    Let $d\ge 3$ be a positive integer and let $G$ be a $K_{1,d}$-free bisplit graph with a bisplit partition $(X,Y,Z)$ of $V(G)$ such that neither $Y$ nor $Z$ is empty.
    Then, $G$ contains no induced cycle of length at least $2d-1$.
\end{lemma}

\begin{proof}
    Let $d\ge 3$ and let $G$ be a $K_{1,d}$-free bisplit graph. Let $(X,Y,Z)$ be a bisplit partition of $V(G)$ such that neither $Y$ nor $Z$ are empty. Since $G$ is $K_{1,d}$-free, it follows that both $Y$ and $Z$ have size at most $d-1$. Now fix $C$ to be any induced cycle in $G$.

    First, suppose that $C$ does not contain both a vertex from $Y$ and a vertex from $Z$. Without loss of generality, we assume that $C$ contains no vertex from $Z$. Since both $Y$ and $X$ are independent sets, it follows that $C$ is an even cycle with exactly half of its vertices belonging to $Y$ and the other half belonging to $X$. Now, since $|Y| < d$, $C$ has length at most $2d-2$.

    Second, we consider the case when $C$ contains both a vertex $y\in Y$ and a vertex $z\in Z$. Since $Y\cup Z$ induces a complete bipartite graph, $yz\in E(G)$. If $y$ and $z$ are the unique vertices of $C$ contained in $Y\cup Z$, then, since $X$ is independent, $C$ is of length at most $3 < 2d - 1$ for all $d\ge 3$.
    Therefore, without loss of generality, we may assume that $C$ contains another vertex from $Y$ distinct from $y$, say $y'$.
    In that case, $C$ contains an induced path $yzy'$. It follows that either $C$ contains another vertex from $Z$, in which case $C$ is a cycle of length $4 < 2d - 1$ for all $d\ge 3$, or $z$ is the only vertex of $C$ from $Z$.
    In that case, since $Y\cup Z$ induces a complete bipartite graph, $C$ cannot contain vertices from $Y\setminus \{y,y'\}$ and $Z\setminus \{z\}$.
    Therefore, $C$ has length at most $4$.
    This concludes the proof that $G$ contains no cycle of length at least $2d-1$.
\end{proof}

Note that in the above lemma if we set $d = 1$, then $G$ is edgeless, and if we set $d = 2$, then $2d-1 = 3$. However, it is easy to observe that $G$ can contain a $C_3$ while not containing any induced cycle of length at least $4$. 

For a graph class $\mathcal{C}$, the \emph{hereditary closure} of $\mathcal{C}$ is the minimal graph class containing all graphs in $\mathcal{C}$ that is closed under taking induced subgraphs.

\begin{lemma}
    Let $d\ge 2$ be a positive integer and let $\mathcal{G}$ be the class of $K_{1,d}$-free bisplit graphs whose bisplit partition $(X,Y,Z)$ of $V(G)$ is such that neither $Y$ nor $Z$ are empty. Let $\hat{\mathcal{G}}$ be the hereditary closure of $\mathcal{G}$. Then, for every $G\in \hat{\mathcal{G}}$, $\atw(G) \leq d-1$.
\end{lemma}
\begin{proof}
    Let $G$ be any graph from $\hat{\mathcal{G}}$.
    Then $G$ is an induced subgraph of a bisplit graph $G'$ with bisplit partition $(X,Y,Z)$ such that neither $Y$ nor $Z$ is empty.
    Furthermore, since $G'$ is $K_{1,d}$-free, we have that $|Y|, |Z| \leq d-1$.
    Then it suffices to show that $\atw(G') \leq d-1$.

    We construct a tree decomposition $(T,\{B_t\}_{t\in V(T)})$ of $G'$ as follows. Let $T$ be a star with $V(T) := \{v_0\} \cup \{v_x \, : \, x\in X\}$ where $v_0$ is the center of the star $T$. Let the bag $B_{v_0}:= Y \cup Z$, and for all $x\in X$ let $B_{v_x}:= N_{G'}[x]$. It follows that $(T, \{B_t\}_{t\in V(T)})$ is a tree decomposition of $G'$, with each bag having independence number at most $d-1$, as desired.
\end{proof}




\section{Outerstring Graphs}\label{sec:outerstring}

In this section, we prove that the class of $K_{1,d}$-free outerstring graphs has bounded tree-independence number.
A graph $G$ is called \emph{outerstring} if $G$ may be represented as an intersection graph of curves, called \emph{grounded curves}, lying inside a unit disc such that for each such curve, one of the endpoints, called \emph{grounded point}, is on the boundary of the disk.

Throughout this section, for a given outerstring graph, without loss of generality, we assume that its intersection model satisfies the following conditions:
\begin{enumerate}
    \item no grounded curve intersects itself,
    \item all grounded points are distinct,
    \item grounded curves intersect the boundary of the disk only at their respective grounded points, and
    \item no three grounded curves intersect a single point.
\end{enumerate}

Now observe that for two grounded curves that do not intersect each other, we can find a curve that bisects the disk, that does not intersect with both grounded curves.
This observation can be extended for arbitrary non-adjacent induced subgraphs of an outerstring graph.

\begin{observation}\label{obs:outerstring_separator}
    Let $G$ be an outerstring graph with a given intersection model of grounded curves in a disk $D$.
    Let $\eta$ be a curve contained in $D$ that does not intersect itself and both endpoints are on the boundary circle of $D$.
    Let $S$ be the set of vertices whose corresponding grounded curves intersect with $\eta$, and let $A$ and $B$ be the set of vertices such that their corresponding grounded curves are contained in two different components of $D-\eta$, respectively.
    Then $S$ is a separator between $A$ and $B$.
\end{observation}

Another main tool for this section is the theorem from \cite{choi2025excludingladderinducedminor}.

\begin{theorem}[\cite{choi2025excludingladderinducedminor}]\label{thm:pathandcycle}
    Let $d\geq 2$ and $\eta\geq d$ be positive integers and let $G$ be a $K_{1,d}$-free graph with $\alpha\text{-}\mathsf{tw}(G)\geq 16d\eta$.
    Then $G$ contains an induced path $P$ and an induced cycle $C$ such that
    \begin{itemize}
        \item $P$ and $C$ are non-adjacent, and
        \item if $S$ is a separator between $P$ and $C$ in $G$, then $\alpha(S)\geq \eta$.
    \end{itemize}
\end{theorem}

Now we show that if an outerstring graph contains an induced cycle $C$, then we can choose at most $6$ vertices from it so that their closed neighbourhood becomes a separator between $C$ and any subgraph non-adjacent to~$C$. 


\begin{lemma}\label{lem:outerstring}
    Let $G$ be an outerstring graph.
    Let $T\subseteq V(G)$ be a set of vertices and let $C$ be an induced cycle in $G-N[T]$. 
    Then there exists $Y\subseteq V(C)$ with $|Y|\leq 6$ such that $N[Y]$ is a separator between $T$ and $V(C)$.
\end{lemma}
\begin{proof}
    Fix an intersection model of grounded curves in a disk $D$ that gives $G$.
    Let $\Gamma_T$ and $\Gamma_C$ be the sets of curves in the intersection model that correspond to $T$ and $V(C)$ respectively.
    Then since $T$ and $V(C)$ are non-adjacent, there exists exactly one connected component of $D-\bigcup_{\gamma\in \Gamma_T\cup \Gamma_C}$ that is incident to a curve in both $\Gamma_T$ and $\Gamma_C$ (see \Cref{fig:outerstring_graphs}).
    Denote this component by $K$.
    Moreover, the boundary of $K$ that lies on the boundary of the disk $D$ consists of exactly two arcs, say $a_1$ and $a_2$.
    Also let $a_3$ be the minimal arc of $D$ that contains all grounded points of curves in $\Gamma_C$.
    Then we can find a curve $\eta$ contained in $K$ satisfying following conditions:
    \begin{itemize}
        \item $\eta$ does not cross itself,
        \item one endpoint of $\eta$ is contained in $a_1$ and the other endpoint is contained in $a_2$,
        \item one connected component of $D-\eta$ contains all grounded curves corresponding to $T$, and
        \item the other connected component of $D-\eta$, contains all grounded curves corresponding to $V(C)$.
    \end{itemize}
    Let $Y\subseteq V(C)$ be the set of vertices whose corresponding grounded curves are incident to $K$ (see \Cref{fig:outerstring_graphs}).
    Note that $N[Y]$ is a separator between $T$ and $V(C)$ (see \Cref{fig:outerstring_graphs}).
    
    We claim that $|Y|\leq 6$.
    Let $u,v\in V(C)$ be the vertices whose corresponding grounded points are endpoints of $a_3$.
    Then we have $u,v\in Y$.
    Suppose that $w\in Y$, and let $\gamma_u,\gamma_v$, and $\gamma_w$ be the grounded curves corresponding to $u,v$, and $w$ respectively.

    We can draw a new curve $\eta'$ that starts from $\eta$ and ends at $\gamma_w$, does not cross itself, and does not intersect with any other grounded curve corresponds to vertices in $V(C)\setminus \{w\}$.
    Then we can find a curve $\nu\subseteq \eta'\cup \gamma_w$, both of whose endpoints are on the boundary of $D$, that contains $\eta'$ and contained in $\eta\cup \eta'\cup \gamma_w$. (See \Cref{fig:outerstring_graphs}.)


    Suppose that both $\gamma_u$ and $\gamma_v$ do not intersect with $\nu$.
    Then $\gamma_u$ and $\gamma_v$ are in distinct components of $D-\nu$.
    This implies that $N_C[w]$ is a separator between $u$ and $v$ in $C$.
    Since $C$ is an induced cycle, this is only possible if either $u$ or $v$ is contained in $N_C[w]$, which is a contradiction.
    Therefore, $\nu$ either intersects $\gamma_u$ or $\gamma_v$, and this is only possible if $\gamma_w$ intersects either $\gamma_u$ or $\gamma_v$.
    In other words, we have $Y\subseteq N_C[u]\cup N_C[v]$, which gives $|Y|\leq 6$.
\end{proof}
    
\begin{figure}[t]
    \centering
    \includegraphics[width=\linewidth]{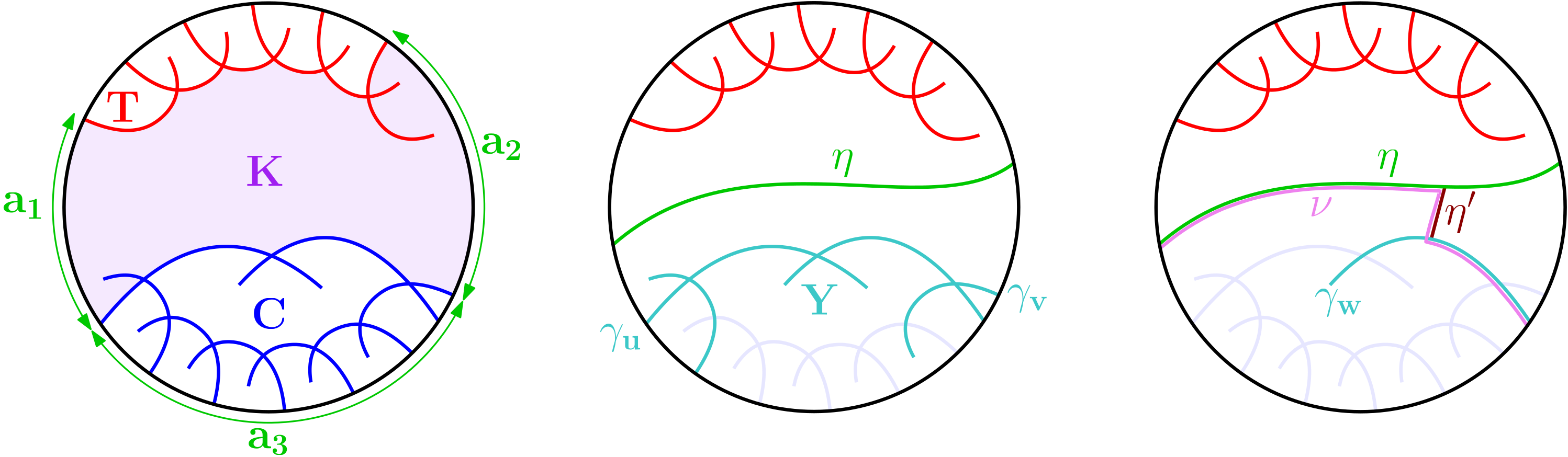}
    \caption{Outerstring graphs in \Cref{lem:outerstring}: (Left) We can find an area $K$ between the grounded curves corresponding to~$T$ and $V(C)$. (Center) A curve $\eta$ bisects the disk so that the grounded curves corresponding to~$T$ and $V(C)$ are contained in distinct components. $Y$ is the set of vertices whose corresponding grounded curves are incident to $K$. (Right) A curve $\nu$ should intersect either $\gamma_u$ or $\gamma_v$.}
    \label{fig:outerstring_graphs}
\end{figure}


Finally, we show that $K_{1,d}$-free outerstring graphs have bounded tree-independence number by using \cref{obs:outerstring_separator}.

\begin{theorem}\label{thm:outerstring}
    Let $G$ be a $K_{1,d}$-free outerstring graph.
    Then $\atw(G)< 16d(6d-5)$.
\end{theorem}
\begin{proof}
    Suppose that $\atw(G)\geq 16d(6d-5)$.
    Then by \cref{thm:pathandcycle}, we can find an induced path $P$ and an induced cycle $C$ such that they are non-adjacent and any separator $S$ between $P$ and $C$ satisfies $\alpha(S)\geq 6d-5$.
    Applying \cref{lem:outerstring}, we can find $Y\subseteq V(C)$ with $|Y|\leq 6$ such that $N[Y]$ is a separator between $P$ and $C$.
    However, $\alpha(N[Y])\leq (d-1)\cdot |Y|\leq 6d-6$, which is a contradiction.
    Therefore, we have $\atw(G)< 16d(6d-5)$.
\end{proof}

Note that An, Oh, and Xue~\cite{AOX2024} showed that if an outerstring graph is planar, then it has treewidth $O(\log n)$.
Since $n\times n$-grid has treewidth $n$, a sufficiently large grid cannot be an outerstring graph.
Therefore, since the class of outerstring graphs is closed under taking induced subgraphs, \cref{thm:outerstring} is a special case for \cref{conj:inducedgridequiv}.

Next, we consider two related graph classes.
A \emph{circle graph} is a graph whose vertices can be associated with a finite set of chords of a circle such that two vertices are adjacent if and only if the corresponding chords cross.
A \emph{circular permutation graph} is a graph which corresponds to an intersection model consisting of paths between two concentric circles, such that no two paths intersect in more than one point.

Observe that circle graphs are outerstring graphs.
This can be seen by slightly shortening each chord of the intersection model, so that the intersection model does not change.
Similarly, the circular permutation graphs are also outerstring graphs, by deleting the inner circle.
Hence, we have the following corollary.

\begin{corollary}
    Let $G$ be a $K_{1,d}$-free graph that is either a circle graph or a circular permutation graph.
    Then $\alpha\text{-}\mathsf{tw}(G)<16d(6d-5)$.
\end{corollary}

Now we consider a generalisation of outerstring graphs such that grounded points could lie on several circles.
A graph $G$ is said to be a $k$-outerstring graph if it is an intersection graph of curves $\gamma\in \Gamma$ such that there exists disjoint circles $C_1,C_2,\ldots,C_k$ on a plane so that none of the circles contains another circle and each curve in $\Gamma$ intersects exactly once one of $C_i$ at one of the endpoints of $\gamma$.
Note that the definition of $1$-outerstring graphs is equal to the definition of outerstring graphs.
Observe that we may change the `circles' in the definition to `closed curves' without affecting the graph class defined.

\begin{theorem}\label{thm:k-outerstring}
    Let $G$ be a $K_{1,d}$-free $k$-outerstring graph.
    Then $\alpha\text{-}\mathsf{tw}(G)< 16d(6d-5)+2(d-1)(k-1)$.
\end{theorem}

\begin{proof}
    We use induction on $k$.
    The base case when $k=1$ follows from \Cref{thm:outerstring}, so we may assume that $k\geq 2$ and our assertion holds for $k-1$.

    Fix an intersection model of $G$ where $C_1,C_2,\ldots,C_k$ are disjoint closed curves on a plane.
    Let $V_i\subseteq V(G)$ be the set of vertices whose grounded points are on $C_i$.
    If $V_k$ is anticomplete to $V(G)\setminus V_k$, then $G-V_k$ is a $(k-1)$-outerstring graph and $G[V_k]$ is $1$-outerstring graph.
    Thus we have $\alpha\text{-}\mathsf{tw}(G)=\max\{\alpha\text{-}\mathsf{tw}(G[V_k]),\alpha\text{-}\mathsf{tw}(G\setminus V_k)\}\leq 16d(6d-5)+2(d-1)(k-2)$.

    Now without loss of generality, suppose that there is an edge between $V_k$ and $V_{k-1}$.
    Then we can find vertices $u\in V_k$ and $v\in V_{k-1}$, whose grounded curves are $\gamma_u$ and $\gamma_v$ respectively, such that there is a curve $\eta$ with one endpoint on $C_k$ and the other on $C_{k-1}$, and $\eta\subseteq \gamma_u\cup \gamma_v$ (see \cref{fig:k-outerstring_graph}).
    We may assume that $\eta$ does not cross itself.
    Notice that the grounded curve corresponding to a vertex $w\in V(G)$ intersects with $\eta$ if and only if $w$ is adjacent to $u$ or $v$.
    
    We can construct a new closed curve $C'$ that encloses $C_k$, $C_{k-1}$, and $\eta$ by adding two curves between $C_k$ and $C_{k-1}$ slightly offset from $\eta$ on either side (see \Cref{fig:k-outerstring_graph}).
    Then $G-(N[u]\cup N[v])$ is a $(k-1)$-outerstring graph, with the closed curves $C_1,C_2,\ldots,C_{k-2},C'$.
    Therefore, we have 
    $$\alpha\text{-}\mathsf{tw}(G)\leq \alpha(N[u]\cup N[v])+\alpha\text{-}\mathsf{tw}(G-(N[u]\cup N[v]))\leq 16d(6d-5)+2(d-1)(k-1)$$
    and the proof is complete.
\end{proof}

\begin{figure}
    \centering
    \includegraphics[width=0.95\linewidth]{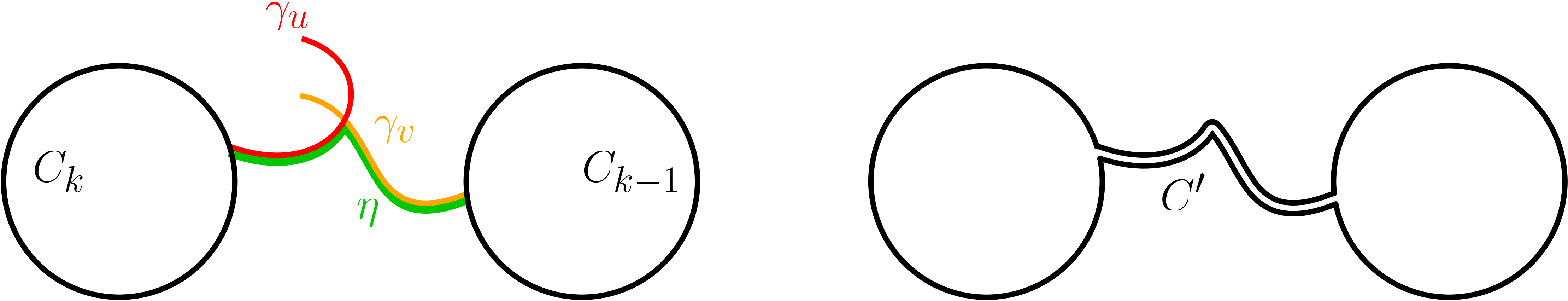}
    \caption{In the proof of \Cref{thm:k-outerstring} we can merge two closed curve using a curve between them.}
    \label{fig:k-outerstring_graph}
\end{figure}

Note that a grid is bipartite and has a tree decomposition witnessing the optimal treewidth such that each bag induces a tree.
Hence the same tree decomposition witnesses the optimal tree-independence number of a grid.

\begin{observation}\label{obs:gridtreealpha}
    For every integer $n\ge 2$, the $n \times n$-grid has tree-independence number $\lceil\frac{n+1}{2}\rceil$. 
\end{observation}

\cref{thm:k-outerstring} also shows that for each $k$, a sufficiently large grid is not a $k$-outerstring graph.

\begin{corollary}
    For each $k\geq 1$, $(16k+3983)\times (16k+3983)$-grid and its line graph are not $k$-outerstring graphs.
\end{corollary}
\begin{proof}
    Let $G=(16k+3983)\times (16k+3983)$-grid.
    Then $\atw(G)=2000+8(k-1)$, and from \cref{thm:k-outerstring} when $d=5$, we conclude that~$G$ is not $k$-outerstring graph.
    The case for the line graph of~$G$ follows similarly.
\end{proof}

\section{Graph classes with unbounded tree-independence number}
\label{sec:unbounded}

In this section we discuss a couple of $K_{1,d}$-free graph classes that have unbounded tree-independence number. Firstly, we note that by \cref{thm:hajebi-spirkl}, the only interesting hereditary classes to consider are those that have an infinite number of forbidden induced subgraphs. Also observe that grids and walls have unbounded tree-independence number~\cite{DBLP:journals/jctb/DallardMS24} (see \cref{obs:gridtreealpha}).
Moreover, grids are $K_{1,5}$-free and walls are $K_{1,4}$-free. Grids and walls are contained in many significant sparse graph classes, such as the classes of planar graphs. 

However, they are also contained in some denser graph classes. For instance, like we show in \cref{sec:circle-containment}, the class of circle containment graphs contains all grids.  Observe that a natural subclass of circle containment graphs, the class of interval containment graphs, in fact has bounded tree-independence number when forbidding $K_{1,d}$. Namely, this is exactly the class of permutation graphs, a subclass of the class of outerstring graphs, and so the result follows from \cref{thm:outerstring}.

Moreover, since subdividing edges does not decrease the tree-independence number (since having tree-independence number at most $k$ is closed under taking induced minors~\cite{DBLP:journals/jctb/DallardMS24}), classes containing subdivisions of arbitrarily large grids or walls also have unbounded tree-independence number. For instance, this gives an easy proof that the class of $\{C_4, K_{1,5}\}$-free graphs has unbounded tree-independence number. We will use the induced minor relation to show that the class of rook graphs has unbounded tree-independence number in \cref{sec:rook}.

\subsection{Circle containment graphs}
\label{sec:circle-containment}
A \emph{circle containment graph} is a graph whose vertex set corresponds to circles (disks) in 2 dimensional space and two vertices are adjacent if and only if one of the corresponding circles is contained in the other one.

\begin{lemma}
    The class of circle containment graphs contains all grids.
\end{lemma}
\begin{proof}
     Consider an $m \times n$ grid $G_{m \times n}$ with $V(G_{m \times n}) := \{(a,b) \, : \, a \in [m], \, b\in [n]\}$ such that $(a_1,b_1)(a_2,b_2)\in E(G_{m \times n})$ if and only if $|a_1 - a_2| + |b_1 - b_2| = 1$. Then, for $\varepsilon$ a fixed small constant, the following is a circle containment representation of $G_{m \times n}$. If for a vertex $(a,b) \in V(G_{m\times n})$ the sum $a +b$ is odd, we represent it by a circle of radius $1 + 2\varepsilon$ with centre point $(a,b)$. Otherwise, if $a + b$ is even, we represent $(a,b)$ by a circle with radius $\varepsilon$ and centre point $(a,b)$. 
\end{proof}

\begin{figure}
    \centering
    \includegraphics[width=0.95\linewidth]{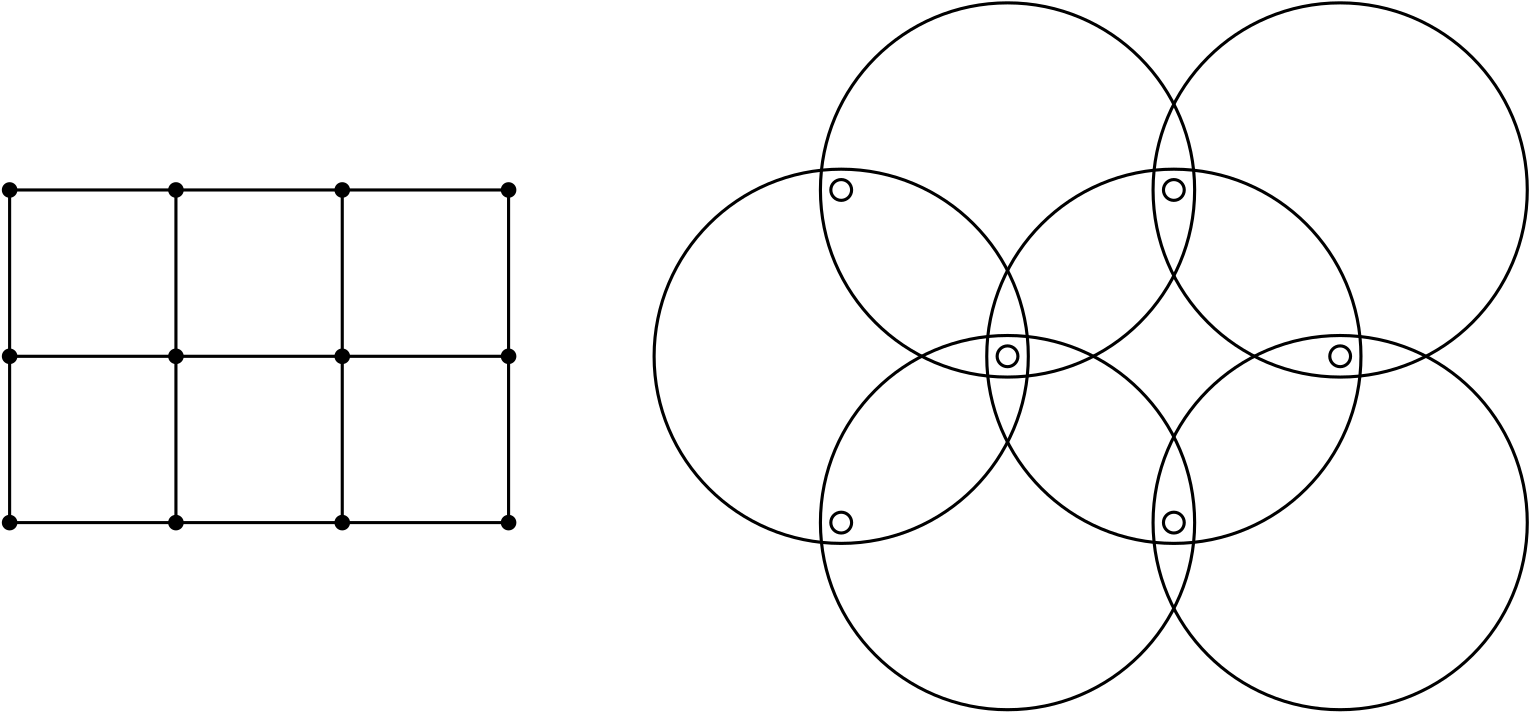}
    \caption{$G_{3 \times 4}$ and its corresponding circle containment representation.}
    \label{fig:circle_containment_grids}
\end{figure}

See \cref{fig:circle_containment_grids} for an illustration of the construction of the proof above. Given that grids are $K_{1,5}$-free and have unbounded tree-independence number~\cite{DBLP:journals/jctb/DallardMS24}, we obtain the following corollary.

\begin{corollary}
    The class of $K_{1,5}$-free circle containment graphs has unbounded tree-independence number.
\end{corollary}



\subsection{Rook graphs}
\label{sec:rook}

For a final example, we consider the class of rook graphs. The $n\times m$ \emph{rook graph} is a graph $R_{n,m}$ with vertex set $V(R_{n,m}):= \{v_{i,j}\, : \, i \in [n], \, j\in [m]\}$ and two vertices $v_{i,j}$ and $v_{i',j'}$ are adjacent if and only if $i=i'$ or $j=j'$. Intuitively, it is the graph given by legal moves for rooks on a $n \times m$ chessboard. As such, they are significantly denser supergraphs of grids. Note that rook graphs are $K_{1,3}$-free.

\begin{lemma}\label{lem:rook_graphs_induced_minors}
    For any graph $H$ there exists an $n_H \in \mathbb{N}^+$ such that $R_{n_H,n_H}$ contains $H$ as an induced minor.
\end{lemma}
\begin{proof}
    We will prove the following more precise statement:
    \begin{claim}
        For any graph $H$ there exists an $n_H \in \mathbb{N}^+$ such that $R_{n_H,n_H}$ contains an induced minor model $\mathbb{X}$ of $H$ where for each $u \in V(H)$ the branch set $X_u \in \mathbb{X}$ is given by $\{v_{i,j}\, :\, i \in I_u, \, j \in J_u\}$ for some subsets $I_u, J_u \subseteq [n_H]$. 
    \end{claim}
    We will proceed by induction on $|V(H)|$. For $|V(H)|=1$, the claim trivially holds with $n_H =1$. Now suppose that $|V(H)|\geq 2$. Let $v \in V(H)$ be an arbitrary vertex and let $H' := H-v$. By induction, the claim holds for $H'$. That is, there exists an $n_{H'} \in \mathbb{N}^+$ and sets $I'_u, J'_u \subseteq [n_{H'}]$ such that $R_{n_{H'},n_{H'}}$ contains an induced minor model $\mathbb{X}'$ of $H'$ where for each $u \in V(H')$ the branch set $X'_u \in \mathbb{X}'$ is given by $\{v_{i,j}\, :\, i \in I'_u, \, j \in J'_u\}$. If $v$ is an isolated vertex in $H$, it suffices to set $n_H:= n_{H'}+1$ and to consider the induced minor model $\mathbb{X}' \cup \{X_v\}$ where $X_v = \{v_{n_H, n_H}\}$. Hence, we may assume that $\deg_H(v) \geq 1$. Then let $v_1, \ldots, v_{\deg_H(v)}$ be the neighbours of $v$ in $H$. Let $n_H := n_{H'} + \deg_H(v)+1$ and consider the induced minor model $\mathbb{X} = \{X_u\}_{u\in V(H)}$ of $H$ in $R_{n_H, n_H}$ defined by sets $I_u, J_u \subseteq [n_H]$ as follows:
    \begin{itemize}
        \item $I_u := I'_u$ and $J_u := J_u'$ for $u \in V(H)\setminus N_H[v]$;
        \item $I_{v_{\ell}} := I'_{v_{\ell}} \cup \{n_{H'} +\ell\}$ for $\ell \in [\deg_H(v)]$ and $J_{v_\ell} := J_{v_\ell}'$; and 
        \item $I_v := \{n_{H'}+1, n_{H'}+2, \ldots, n_{H'}+\deg_H(v)\}$ and $J_v := \{n_H\}$.
    \end{itemize}
    Since two vertices $u,u'\in V(H)$ are adjacent according to the induced minor model if and only if $I_u \cap I_{u'} \neq \emptyset$ or $J_u \cap J_{u'}\neq \emptyset$, the correctness of this induced minor model follows from the correctness of the induced minor model $\mathbb{X}'$ of $H'$ and the observation that the extension of $\mathbb{X}'$ to $\mathbb{X}$ is exactly responsible for the edges incident to $v$ in $H$.
\end{proof}
Given that having tree-independence number at most $k$ is closed under taking induced minors~\cite{DBLP:journals/jctb/DallardMS24} and that graph classes of unbound tree-independence number exist, \cref{lem:rook_graphs_induced_minors} implies the following:
\begin{corollary}
    The class of rook graphs has unbounded tree-independence number.
\end{corollary}

\section{A step towards $K_{2,d}$-free graph classes}
\label{sec:K2d}

Finally, we now show that, in fact, when excluding all holes of length at least $5$, one can prove a slightly stronger result than the one from Theorem~\ref{thm:atw_K1d-free_longest_cycle}. We start with some definitions following terminology from \cite{bousquet2025inducedminormodelsi}.

Let $G$ and $H$ be two graphs such that $H$ is an induced minor of $G$. An \emph{induced minor model} of $H$ in $G$ is a collection $\mathbb{X} = \{X_u: u\in V(H)\}$ of subsets of $V(G)$, called \emph{branch sets} of $\mathbb{X}$, such that
\begin{itemize}
    \item for every $u,v\in V(H)$, $u\not = v$, $X_u\cap X_v = \emptyset$,
    \item for every $u\in V(H)$, the graph $G[X_u]$ is connected, and
    \item for every $u,v\in V(H)$, $u\not = v$, there exists an edge in $G$ between $X_u$ and $X_v$ if and only if $uv\in E(H)$.
\end{itemize}
The subgraph of $G$ \emph{induced by} $\mathbb{X}$ is the subgraph of $G$ induced by the union of the branch sets of $\mathbb{X}$. An induced minor model $\mathbb{X} = \{X_u: u\in V(H)\}$ of $H$ in $G$ is said to be \emph{minimal} if $\mathbb{X}$ induced a subgraph $G'$ of $G$ such that no proper induced subgraph of $G'$ contains $H$ as an induced minor.

Before we state our results we present the following structural lemma for minimal induced minor models.

\begin{lemma}[\cite{bousquet2025inducedminormodelsi}]\label{lem:induced_minor_model_properties}
    Let $G$ and $H$ be two graphs such that $H$ is an induced minor of $G$. Then, there exists a minimal induced minor model $\mathbb{X} = \{X_u: u\in V(H)\}$ of $H$ in $G$ such that the following holds.
    \begin{enumerate}
        \item $|X_u| = 1$ for every vertex $u\in V(H)$ with degree at most $2$ such that $u$ belongs to a connected component of $H$ that is not a cycle.
        \item For every connected component $C$ of $H$ that is a cycle, there exists a vertex $v$ of $C$ such that $G[X_v]$ is a path and $|X_u| = 1$ for every vertex $u$ in $V(C - v)$.
    \end{enumerate}
\end{lemma}

Using Lemma~\ref{lem:induced_minor_model_properties}, we can now prove the following.

\begin{theorem}
    \label{thm:no_long_hole_Kpq}
    Let $G$ be a graph with no holes of length at least $5$. Let $p,q\ge 2$ be integers. Then $G$ contains $K_{p,q}$ as an induced minor if and only if $G$ contains $K_{p,q}$ as an induced subgraph.
\end{theorem}

\begin{proof}
    One direction is trivial: if $G$ contains $K_{p,q}$ as an induced subgraph, then $G$ contains $K_{p,q}$ as an induced minor.

    To prove the other direction, let $G$ be a graph with no holes of length at least $5$ containing $K_{p,q}$ as an induced minor.
    Let $\mathbb{X} = \{A_1,\ldots , A_p, B_1,\ldots , B_q\}$ be a minimal induced minor model of $K_{p,q}$ in $G$ with a bipartition $(\mathcal{A}, \mathcal{B})$ where $\mathcal{A} := \{A_1,\ldots , A_p\}$ and $\mathcal{B} := \{B_1,\ldots ,B_q\}$.
    If every bag of $\mathbb{X}$ contains a single vertex, then we are done, so there is a bag of $\mathbb{X}$ containing at least two vertices.
    Without loss of generality, suppose that $|A_1|\geq 2$. 
    For each $u\in A_1$, let $\phi(u)=|\{B\in \mathcal{B}\mid N(u)\cap B\neq \emptyset \}|$ be the number of bags of $\mathcal{B}$ that contain a vertex adjacent to $u$.
    Let $u$ be a vertex of $A_1$ maximizing $\phi(u)$.

    We now prove that we can find an induced cycle of $G$ of length at least $5$, which is a contradiction.

    By the minimality of $\mathbb{X}$, there must be a bag $B_i\in \mathcal{B}$ such that $N(u)\cap B_i=\emptyset$.
    Without loss of generality, let $i = q$.
    It follows that there exist a vertex in $A_1$, distinct from $u$, adjacent to $B_q$.
    Choose a vertex $v\in A_1\cap N(B_q)$ such that the distance between $u$ and $v$ in $A_1$ is minimized, and let $P$ be a shortest $u$-$v$ path in $A_1$.
    Note that $v$ is the only vertex of $P$ adjacent to $B_q$.

    By the maximality of $\phi(u)$, there exists $j\in \{1,\ldots,q-1\}$ such that $u$ is adjacent to $B_j$ while $v$ is not adjacent to $B_j$.
    Let $w\neq v$ be the vertex on $P$ adjacent to $B_j$ that minimizes the distance to $v$ along $P$.
    Note that $u=w$ is also possible.
    Then we can find a $v$-$w$ path $Q$ whose internal vertices are contained in  $B_j\cup A_2\cup B_q$.
    Concatenating $Q$ with the subpath of $P$ between $v$ and $w$ yields an induced cycle of length at least $5$, which is contradiction.
    Therefore, we conclude that each bag of $\mathbb{X}$ contains exactly one element, so the vertices contained in the induced minor model of $K_{p,q}$ in $G$ also induces $K_{p,q}$ as an induced subgraph.
\end{proof}

Theorem~\ref{thm:no_long_hole_Kpq} immediately implies the following.

\begin{corollary}
    \label{cor:no_K2d_induced_minor}
    Let $G$ be a $K_{1,d}$-free graph containing no holes of length at least $5$. Then $G$ is a \hbox{$K_{2,d}$-induced-minor-free graph}.
\end{corollary}

\begin{proof}
    Suppose to the contrary that $G$ contains a $K_{2,d}$ as an induced minor. Then, by Theorem~\ref{thm:no_long_hole_Kpq} $G$ contains a $K_{2,d}$ as an induced subgraph, a contradiction since $G$ is $K_{1,d}$-free.
\end{proof}

Before we prove our final result of this section, we need to state the following result by Dallard, Milanič, and Štorgel~\cite{dallard2024treewidthversuscliquenumber3}.

\begin{theorem}[\cite{dallard2024treewidthversuscliquenumber3}]
    \label{thm:alpha_tw_K2q_induced_minor_free}
    For every integer $q\ge 2$ and every $K_{2,q}$-induced-minor-free graph $G$, the tree-independence number of $G$ is at most $2q - 2$.
\end{theorem}

Using Theorem~\ref{thm:no_long_hole_Kpq} together with Theorem~\ref{thm:alpha_tw_K2q_induced_minor_free}, we are now ready to state out result which is a generalisation of Theorem~\ref{thm:atw_K1d-free_longest_cycle} in the case when $\ell = 4$.

\begin{theorem}
    \label{thm:K2d-main}
    Let $d\ge 2$ be an integer. Let $G$ be a $K_{2,d}$-free graph containing no holes of length at least $5$. Then $\atw (G) \leq 2d-2$.
\end{theorem}

\begin{proof}
    Let $G$ be a $K_{2,d}$-free graph containing no holes of length at least $5$. By Theorem~\ref{thm:no_long_hole_Kpq}, since~$G$ is $K_{2,d}$-free, $G$ contains no $K_{2,d}$ as an induced-minor. It follows, by Theorem~\ref{thm:alpha_tw_K2q_induced_minor_free}, that $G$ has tree-independence number at most $2d-2$ as claimed.
\end{proof}

Note that \cref{thm:K2d-main} generalises both \cref{no_long_hole_tolerance} and \cref{no_long_hole_co-comparability}, and more generally, it generalises \cref{thm:atw_K1d-free_longest_cycle} in the case when $\ell = 4$.



\section{Open problems}
\label{sec:conclusion}

We conclude by presenting several open questions.
First, observe that the bounds on tree-independence number shown in \cref{sec:without_long_cycle} and \cref{sec:outerstring} are given in polynomial in $d$.
Thus we may ask the following strengthening of \cref{conj:inducedgridequiv} that whether the bound on tree-independence number of a $K_{1,d}$-free graph class without a grid as an induced minor is always given by a polynomial.

\begin{conjecture}
    Is there a polynomial $f$ such that for any integer $d\geq 2$, if a $K_{1,d}$-free graph $G$ does not contain a $d\times d$ grid as an induced minor, then $\atw(G)\leq f(d)$?
\end{conjecture}

Comparing width-parameters in various graph classes play an important role both in structural and algorithmic graph theory.
In \cite{BMPY2023}, authors verified relationships between width parameters of several graph classes, where the only open case is relation between tree-independence number and sim-width in the $K_{t,t}$-free graph class.
One may ask a weaker version of above question by forbidding $K_{1,d}$ instead of $K_{t,t}$.

\begin{conjecture}\label{ques:sim-width}
    For each $d \geq 2$, does a $K_{1,d}$-free graph class with bounded sim-width have bounded tree-independence number?
\end{conjecture}

Another interesting width parameter to compare is twin-width.
In \cite{AJKPRW2025}, authors proved that the graphs with twin-width $1$ are permutation graphs.
Hence by \cref{thm:outerstring}, the class of $K_{1,d}$-free graphs with twin-width~$1$ has bounded tree-independence number.
However, an $n\times n$ grid where $n\geq 7$ has twin-width $4$, so the class of graphs with twin-width at most $4$ has unbounded tree-independence number.
Therefore, one may ask what happens in the case of classes of graphs of twin-width $2$ or $3$.

\begin{conjecture}\label{ques:twin-width}
    For each $d \geq 2$, does a $K_{1,d}$-free graph class with twin-width at most $3$ have bounded tree-independence number?
\end{conjecture}


\section{Acknowledgements}
This work was done while visiting University of Primorska, FAMNIT, Slovenia for the Koper-Leipzig Graph Theory Workshop in 2025.

\bibliographystyle{alphaurl}
\bibliography{Reference}

\end{document}